\documentclass[12pt]{article}
\usepackage{amsmath, amsthm}
\usepackage{enumerate}
\usepackage{amssymb}
\usepackage{geometry}
\newtheorem{theorem}{Theorem}[section]

\newtheorem{corollary}[theorem]{Corollary}

\newtheorem{definition}[theorem]{Definition}
\newtheorem{example}[theorem]{Example}

\newtheorem{lemma}[theorem]{Lemma}

\newtheorem{proposition}[theorem]{Proposition}
\newtheorem{remark}[theorem]{Remark}

\DeclareMathOperator{\fc}{\xrightarrow[]{mo}}

\DeclareMathOperator{\mc}{\xrightarrow[]{mn}}

\DeclareMathOperator{\uoc}{\xrightarrow[]{uo}}
\DeclareMathOperator{\oc}{\xrightarrow[]{o}}
\DeclareMathOperator{\nc}{\xrightarrow{\lVert \cdot \rVert}}

\newcommand{\eval}[2][\right]{\relax
 \ifx#1\right\relax \left.\fi#2#1\rvert}

\begin{document}

\title{\bf Multiplicative norm convergence in Banach lattice $f$-algebras} 
\maketitle

\author{\centering Abdullah Ayd{\i}n \\ \bigskip \small 
	Department of Mathematics, Mu\c{s} Alparslan University, Mu\c{s}, Turkey \\}

\bigskip

\abstract{A net $(x_\alpha)$ in an $f$-algebra $E$ is called multiplicative order convergent to $x\in E$ if $\lvert x_\alpha-x\rvert u\oc 0$ for all $u\in E_+$. This convergence has been investigated and applied in a recent paper by Ayd{\i}n \cite{AAydn}. In this paper, we study a variation of this convergence for Banach lattice $f$-algebras. A net $(x_\alpha)$ in a Banach lattice $f$-algebra $E$ is said to be multiplicative norm convergent to $x\in E$ if $\lvert \lvert x_\alpha-x\rvert u\rVert\to 0$ for each $u\in E_+$. We study on this concept and investigate its relationship with the other convergences, and also we introduce the $mn$-topology Banach lattice $f$-algebras.}

\bigskip
\let\thefootnote\relax\footnotetext
{
	
\text{2010 AMS Mathematics Subject Classification:} Primary 46A40; Secondary 46E30.

e-mail: a.aydin@alparslan.edu.tr}

\section{Introductory Facts}
Let us recall some notations and terminologies used in this paper. An ordered vector space $E$ is said to be {\em vector lattice} (or, {\em Riesz space}) if, for each pair of vectors $x,y\in E$, the supremum $x\vee y=\sup\{x,y\}$ and the infimum $x\wedge y=\inf\{x,y\}$ both exist in E. For $x\in E$, $x^+:=x\vee 0$, $x^-:=(-x)\vee0$, and $\lvert x\rvert:=x\vee(-x)$ are called the {\em positive} part, the {\em negative} part, and the {\em absolute value} of $x$, respectively. A vector lattice $E$ is called \textit{order complete} if $0\leq x_\alpha\uparrow\leq x$ implies the existence of $\sup{x_\alpha}$. A partially ordered set $A$ is called {\em directed} if, for each $a_1,a_2\in A$, there is another $a\in A$ such that $a\geq a_1$ and $a\geq a_2$ (or, $a\leq a_1$ and $a\leq a_2$). A function from a directed set $A$ into a set $E$ is called a {\em net} in $E$. A net $(x_\alpha)_{\alpha\in A}$ in a vector lattice $E$ is\textit{ order convergent} (or \textit{$o$-convergent}, for short) to $x\in E$, if there exists another net $(y_\beta)_{\beta\in B}$ satisfying $y_\beta \downarrow 0$, and for any $\beta\in B$ there exists $\alpha_\beta\in A$ such that $|x_\alpha-x|\leq y_\beta$ for all $\alpha\geq\alpha_\beta$. In this case, we write $x_\alpha\xrightarrow{o} x$. A vector $e\geq 0$ in a vector lattice $E$ is said to be a \textit{weak order unit} whenever the band generated by $e$ satisfies $B_e=E$, or equivalently, whenever for each $x\in E_+$ we have $x\wedge ne\uparrow x$; see more information for example (\cite{AB}, \cite{ABPO},\cite{AAyd},\cite{AGG}, \cite{Vul}, \cite{Za}).

A vector lattice $E$ under an associative multiplication is said to be a \textit{Riesz algebra} whenever the multiplication makes $E$ an algebra (with the usual properties), and in addition, it satisfies the following property: $xy\in E_+$ for every $x,y\in E_+$. A Riesz algebra $E$ is called \textit{commutative} if $xy=yx$ for all $x,y\in E$. A Riesz algebra $E$ is called \textit{$f$-algebra} if $E$ has additionally property that $x\wedge y=0$ implies $(xz)\wedge y=(zx)\wedge y=0$ for all $z\in E_+$. A vector lattice $E$ is called \textit{Archimedean} whenever $\frac{1}{n}x\downarrow 0$ holds in $E$ for each $x\in E_+$. Every Archimedean $f$-algebra is commutative; see Theorem 140.10 \cite{Za}. Assume $E$ is an Archimedean $f$-algebra with a multiplicative unit vector $e$. Then, by applying Theorem 142.1(v) \cite{Za}, in view of $e=ee=e^2\geq0$, it can be seen that $e$ is a positive vector. On the other hand, since $e\wedge x=0$ implies $x=x\wedge x= (xe)\wedge x= 0$, it follows that $e$ is a weak order unit. In this article, unless otherwise, all vector lattices are assumed to be real and Archimedean, and so $f$-algebras are commutative.

Recall that a net $(x_\alpha)$ in an $f$-algebra $E$ is called {\em multiplicative order convergent} (or shortly, {\em $mo$-convergent}) to $x\in E$ if $\lvert x_\alpha-x\rvert u\oc 0$ for all $u\in E_+$. Also, it is called {\em $mo$-Cauchy} if the net $(x_\alpha-x_{\alpha'})_{(\alpha,\alpha') \in A\times A}$ $mo$-converges to zero. $E$ is called {\em $mo$-complete} if every $mo$-Cauchy net in $E$ is $mo$-convergent, and is called {\em $mo$-continuous} if $x_\alpha\oc0$ implies $x_\alpha\fc 0$; see for detail information \cite{AAydn}. On the other hand, a net $(x_\alpha)$ in a Banach lattice $E$ is {\em unbounded norm convergent} (or {\em un-convergent}) to $x\in E$ if $\lvert \lvert x_\alpha-x\rvert \wedge u\rVert\to0$ for all $u\in E_+$; see \cite{Try}. We routinely use the following fact: $y\leq x$ implies $uy\leq ux$ for all positive elements $u$ in $f$-algebras. Moreover, an $f$-algebra $E$ which is at the same time a Banach lattice is called a {\em Banach lattice $f$-algebra} whenever $\lvert xy\rVert\leq \Vert x\rVert\Vert y\rVert$ holds for all $x,y\in E$. Motivated from above definitions, we give the following notion.

\begin{definition}
	A net $(x_\alpha)$ in a Banach lattice $f$-algebra $E$ is said to be {\em multiplicative norm convergent} (or shortly, {\em $mn$-convergent}) to $x\in E$ if $\lvert \lvert x_\alpha-x\rvert u\rVert\to 0$ for all $u\in E_+$. Abbreviated as $x_\alpha\mc x$. If the condition holds only for sequences then it is called sequentially $mn$-convergence.
\end{definition}

\begin{remark}\ \label{rem 1 basic}
	\begin{enumerate}
		\item[(i)] For a net $(x_\alpha)$ in a Banach lattice $f$-algebra $E$, $x_\alpha\mc x$ implies $x_\alpha y\mc xy$ for all $y\in E$ because of $\lVert\lvert x_\alpha y-xy\rvert u\rVert=\lVert \lvert x_\alpha-x\rvert\lvert y\rvert u\rVert$ for all $u\in E_+$. The converse holds true in Banach lattice $f$-algebras with the multiplication unit. Indeed, assume $x_\alpha y\mc xy$ for each $y\in E$. Fix $u\in E_+$. So, $\lVert \lvert x_\alpha-x\rvert u\rVert=\lVert \lvert x_\alpha e-xe\rvert u\rVert\mc 0$. 
		
		\item[(ii)] In Banach lattice $f$-algebras, the norm convergence implies the $mn$-convergence. Indeed, by considering the inequality $\lVert \lvert x_\alpha-x\rvert u\rVert\leq \lVert x_\alpha-x\rVert\lVert u\rVert$ for any net $x_\alpha\mc x$, we can get the desired result.
		
		\item[(iii)] If a net $(x_\alpha)$ is order Cauchy and $x_\alpha\mc x$ in a Banach lattice $f$-algebra then we have $x_\alpha\fc x$. Indeed, since order Cauchy norm convergent net is order convergent to its norm limit, we can get the desired result.
		
		\item[(iv)] In order continuous Banach lattice $f$-algebras, both the (sequentially) order convergence and the (sequentially) $mo$-convergence imply the (sequentially) $mn$-convergence.
		
		\item[(v)] In atomic and order continuous Banach lattice $f$-algebras, a sequence which is order bounded and $mn$-convergent to zero is sequentially $mo$-convergent to zero; see Lemma 5.1 \cite{DOT}.
		
		\item[(vi)] For an $mn$-convergent to zero sequence $(x_n)$ in a Banach lattice $f$-algebra, there is a subsequence $(x_{n_k})$ which sequentially $mo$-converges to zero; see Lemma 3.11 \cite{GX}.
	\end{enumerate}
\end{remark} 

\begin{example}\label{int exam}
	Let $E$ be a Banach lattice. Fix an element $x\in E$. Then the principal ideal $I_x=\{y\in E: \exists \lambda>0 \ \text{with} \ \lvert y\rvert\leq \lambda x\}$, generated by $x$ in $E$ under the norm $\lVert\cdot\rVert_\infty$ which is defined by	$\lVert y\rVert_\infty=\inf\{\lambda>0:\lvert y\rvert\leq \lambda x\}$, is an $AM$-space; see Theorem 4.21 \cite{ABPO}. 
	
	Recall that a vector $e>0$ is called order unit whenever for each $x$ there exists some $\lambda>0$ with $\lvert x\rvert\leq \lambda e$. Thus, we have $(I_x,\lVert\cdot\rVert_\infty)$ is $AM$-space with the unit $\lvert x\rvert$. Since every $AM$-space with unit, besides being a Banach lattice, has also an $f$-algebra structure. So, we can say that $(I_x,\lVert\cdot\rVert_\infty)$ is a Banach lattice $f$-algebra. Therefore, for a net $(x_\alpha)$ in $I_x$ and $y\in I_x$, by applying Corollary 4.4 \cite{ABPO}, we get $x_\alpha \mc y$ in the original norm of $E$ on $I_x$ iff $x_\alpha\mc y$ in the norm $\lVert\cdot\rVert_\infty$. In particular, take $x$ as the unit element $e$ of $E$. Then we have $E_e=E$. Thus, for a net $(x_\alpha)$ in $E$, we have $x_\alpha \mc y$ in the $(E,\lVert\cdot\rVert_\infty)$ iff $x_\alpha \mc y$ in the $(E,\lVert\cdot\rVert)$.
\end{example}
\section{Main Results}
We begin the section with the next list of properties of $mn$-convergence which follows directly from the inequalities $\lvert x-y\rvert \leq \lvert x-x_\alpha\rvert +\lvert x_\alpha-y\rvert $ and $\lvert \lvert x_\alpha\rvert -\lvert x\rvert \rvert \leq\lvert x_\alpha-x\rvert $.
\begin{lemma}
	Let $(x_\alpha)$ and $(y_\alpha)$ be two nets in a Banach lattice $f$-algebra $E$. Then the following holds true:
	\begin{enumerate}
		\item[(i)] $x_\alpha\mc x$ iff $(x_\alpha- x)\mc 0$ iff $\lvert x_\alpha-x\rvert \mc 0$;
		\item[(ii)] if $x_\alpha\mc x$ then $y_\beta\mc x$ for each subnet $(y_\beta)$ of $(x_\alpha)$;
		\item[(iii)] suppose $x_\alpha\mc x$ and $y_\beta\mc y$, then $ax_\alpha+by_\beta\mc ax+by$ for any $a,b\in \mathbb{R}$;
		\item[(iv)] if $x_\alpha \mc x$ and $x_\alpha \mc y$ then $x=y$; 
		\item[(v)] if $x_\alpha \mc x$ then $\lvert x_\alpha\rvert \mc \lvert x \rvert$.
	\end{enumerate}
\end{lemma}

The lattice operations in Banach lattice $f$-algebras are $mn$-continuous in the following sense.
\begin{proposition}\label{LO are $mn$-continuous}
	Let $(x_\alpha)_{\alpha \in A}$ and $(y_\beta)_{\beta \in B}$ be two nets in a Banach lattice $f$-algebra $E$. If $x_\alpha\mc x$ and $y_\beta\mc y$ then $(x_\alpha\vee y_\beta)_{(\alpha,\beta)\in A\times B} \mc x\vee y$. In particular, $x_\alpha^+\mc x^+$.
\end{proposition}

\begin{proof}
	Assume $x_\alpha\mc x$ and $y_\beta\mc y$. Then, for a given $\varepsilon>0$, there exist indexes $\alpha_0\in A$ and $\beta_0\in B$ such that $\lVert\lvert x_\alpha-x\rvert u\rVert\leq \frac{1}{2}\varepsilon$ and $\lVert\lvert y_\beta-y\rvert u\rVert\leq \frac{1}{2}\varepsilon$ for every $u\in E_+$ and for all $\alpha\geq\alpha_0$ and $\beta\geq\beta_0$. It follows from the inequality $\vert a\vee b-a\vee c\rvert\leq \lvert b-c\rvert$ in vector lattices that 	
	\begin{equation*}
	\begin{split}
	\lVert\lvert x_\alpha \vee y_\beta - x\vee y\rvert u\rVert&\leq \lVert\lvert x_\alpha \vee y_\beta -x_\alpha \vee y\rvert u+\lvert x_\alpha \vee y- x\vee y\rvert u\rVert\\ &\leq \lVert\lvert y_\beta -y\rvert u\rVert+\lVert\lvert x_\alpha-x\rvert u\rVert\leq \frac{1}{2}\varepsilon+\frac{1}{2}\varepsilon=\varepsilon
	\end{split}
	\end{equation*}
	for all $\alpha\geq\alpha_0$ and $\beta\geq\beta_0$ and for every $u\in E_+$. That is, $(x_\alpha\vee y_\beta)_{(\alpha,\beta)\in A\times B} \mc x\vee y$.
\end{proof}

The following proposition is similar to Proposition 2.7 \cite{AAydn}, and so we omit its proof.
\begin{proposition}
	Let $E$ be a Banach lattice $f$-algebra, $B$ be a projection band of $E$ and $P_B$ be the corresponding band projection. Then $x_\alpha\mc x$ in $E$ implies $P_B(x_\alpha)\mc P_B(x)$ in both $E$ and $B$.
\end{proposition}

A positive vector $e$ in a normed vector lattice $E$ is called {\em quasi-interior point} iff $\lVert x-x\wedge ne\rVert\to 0$ for each $x\in E_+$. If $(x_\alpha)$ is a net in a vector lattice with a weak unit $e$ then $x_\alpha \uoc 0$ iff $\lvert x_\alpha\rvert\wedge e\oc 0$; see Lemma 3.5 \cite{GTX}, and for quasi-interior point case; see Lemma 2.11 \cite{DOT}. Analogously, we show the next result.
\begin{proposition}\label{quasi-interior point}
	Let $0\leq(x_\alpha)_{\alpha\in A}\downarrow$ be a net in a Banach lattice $f$-algebra $E$ with a quasi-interior point $e$. Then $x_\alpha\mc 0$ iff $(x_\alpha e)$ norm converges to zero.
\end{proposition}

\begin{proof}
	The forward implication is immediate because of $e\in E_+$. For the converse implication, fix a positive vector $u\in E_+$ and $\varepsilon>0$. Thus, for a fixed index $\alpha_1$, we have	$x_\alpha\leq x_{\alpha_1}$ for all $\alpha\geq \alpha_0$. Then
	$$
	x_\alpha u\leq x_\alpha(u-u\wedge ne)+x_\alpha(u\wedge ne)\leq x_{\alpha_1}(u-u\wedge ne)+nx_\alpha e
	$$
	for all $\alpha\geq\alpha_1$ and each $n\in \mathbb{N}$. Hence, we get
	$$
	\lVert x_\alpha u\rVert\leq \lVert x_{\alpha_1}\rVert\lVert u-u\wedge ne\rVert+n\lVert x_\alpha e\rVert
	$$
	for every $\alpha\geq\alpha_1$ and each $n\in \mathbb{N}$. So, we can find $n$ such that $\lVert u-u\wedge ne\rVert<\frac{\varepsilon}{2\lVert x_{\alpha_1}\rVert}$ because $e$ is a quasi-interior point. On the other hand, it follows from $x_\alpha e\nc 0$ that there exists an index $\alpha_2$ such that $\lVert x_\alpha e\rVert<\frac{\varepsilon}{2n}$ whenever $\alpha\geq\alpha_2$. Since index set $A$ is directed set, there exists another index $\alpha_0\in A$ such that $\alpha_0\geq \alpha_1$ and $\alpha_0\geq \alpha_2$. Therefore, we get $\lVert x_\alpha u\rVert<\lVert x_{\alpha_0}\rVert\frac{\varepsilon}{2\lVert x_{\alpha_0}\rVert}+n\frac{\varepsilon}{2n}=\varepsilon$, and so $\lVert x_\alpha u\rVert\to 0$.
\end{proof}

\begin{remark}
	A positive and decreasing net $(x_\alpha)$ in an order continuous Banach lattice $f$-algebra $E$ with weak unit $e$ is $mn$-convergent to zero iff $x_\alpha e\nc 0$. Indeed, it is known that $e$ is a weak unit iff $e$ is a quasi-interior point in an order continuous Banach lattice. Thus, by applying Proposition \ref{quasi-interior point}, we can get the desired result.
\end{remark}

\begin{proposition}\label{mn convergence positive}
	Let $(x_\alpha)$ be a net in a Banach lattice $f$-algebra $E$. Then we have that
	\begin{enumerate}
		\item[(i)] $0\leq x_\alpha\mc x$ implies $x\in E_+$,
		\item[(ii)] if $(x_\alpha)$ is monotone and $x_\alpha\mc x$ then $x_\alpha\oc x$.
	\end{enumerate}
\end{proposition}

\begin{proof}
	$(i)$ Assume $(x_\alpha)$ consists of non-zero elements and $mn$-converges to $x\in E$. Then, by Proposition \ref{LO are $mn$-continuous}, we have $x_\alpha=x_\alpha^+\mc x^+=0$. Therefore, we get $x\in E_+$.
	
	$(ii)$ For order convergence of $(x_\alpha)$, it is enough to show that $x_\alpha\uparrow$ and $x_\alpha\mc x$ implies $x_\alpha\oc x$. For a fixed index $\alpha$, we have $x_\beta-x_\alpha\in X_+$ for $\beta\ge\alpha$. By applying $(i)$, we can see $x_\beta-x_\alpha\mc x-x_\alpha\in X_+$ as $\beta \to \infty$. Therefore, $x\geq x_\alpha$ for the index $\alpha$. Since $\alpha$ is arbitrary, $x$ is an upper bound of $(x_\alpha)$. Assume $y$ is another upper bound of $(x_\alpha)$, i.e., $y\geq x_\alpha$ for all $\alpha$. So, $y-x_\alpha\mc y-x\in X_+$, or $y\ge x$, and so $x_\alpha \uparrow x$.
\end{proof}

The $mn$-convergence passes obviously to any Banach lattice sub-$f$-algebra $Y$ of a Banach lattice $f$-algebra $E$, i.e., for any net $(y_\alpha)$ in $Y$, $y_\alpha\mc0$ in $E$ implies $y_\alpha\mc0$ in $Y$. For the converse, we give the following theorem which is similar to Theorem 2.10 \cite{AAydn}.
\begin{theorem}\label{$up$-regular}
	Let $Y$ be a Banach lattice sub-$f$-algebra of an Banach lattice $f$-algebra $E$ and $(y_\alpha)$ be a net in $Y$. If $y_\alpha\mc 0$ in $Y$ then it $mn$-converges to zero in $E$ for both of the following cases hold;
	\begin{enumerate}
		\item[(i)] $Y$ is majorizing in $E$;
		\item[(ii)] $Y$ is a projection band in $E$;
	\end{enumerate}	
\end{theorem}

It is known that every Archimedean vector lattice has a unique order completion; see Theorem 2.24 \cite{ABPO}.
\begin{theorem}
	Let $E$ and $E^\delta$ be Banach lattice $f$-algebras with $E^\delta$ being order completion of $E$. Then, for a sequence $(x_n)$ in $E$, the followings hold true:
	\begin{enumerate}
		\item[(i)] If $x_n\mc 0$ in $E$ then there is a subsequence $(x_{n_k})$ of $(x_n)$ such that $x_{n_k}\fc 0$ in $E^\delta$;
		\item[(ii)]	If $x_n\mc 0$ in $E^\delta$ then there is a subsequence $(x_{n_k})$ of $(x_n)$ such that $x_{n_k}\fc 0$ in $E$.
	\end{enumerate}
\end{theorem}

\begin{proof}
	Let $x_n\mc 0$ in $E$, i.e., $\lvert x_n\rvert u\nc 0$ in $E$ for all $u\in E_+$. Now, let's fix $v\in E^\delta_+$. Then there exists $u_v\in E_+$ such that $v\leq u_v$ because $E$ majorizes $E^\delta$. Since $\lvert x_n\rvert u_v\nc 0$, by the standard fact in Exercise 13 \cite[p.25]{AB}, there exists a subsequence $(x_{n_{k}})$ of $(x_n)$ such that $(\lvert x_{n_{k}}\rvert u_v)$ order converges to zero in $E$. Thus, $\lvert x_{n_k}\rvert u_v\oc 0$ in $E^\delta$; see Corollary 2.9 \cite{GTX}. Then it follows from the inequality $\lvert x_{n_k}\rvert v\leq \lvert x_{n_k}\rvert u_v$ that we have $\lvert x_{n_k}\rvert v\oc 0$ in $E^\delta$. That is, $x_{n_k}\fc 0$ in the order completion $E^\delta$ because $v\in E^\delta_+$ is arbitrary.\\
	
	For the converse, put $x_n\mc 0$ in $E^\delta$. Then, for all $u\in E^\delta_+$, we have $\lvert x_n\rvert u\nc 0$ in $E^\delta$. In particular, for all $w\in E_+$, $\lVert\lvert x_n\rvert w\rVert\to 0$ in $E^\delta$. Fix $w\in E_+$. Then, again by the standard fact in Exercise 13 \cite[p.25]{AB}, we have a subsequence $(x_{n_{k}})$ of $(x_n)$ such that $(x_{n_{k}})$ is order convergent to zero in $E^\delta$. Thus, by Corollary 2.9 \cite{GTX}, we see  $\lvert x_{n_k}\rvert w\oc 0$ in $E$. As a result, since $w$ is arbitrary, $x_{n_k}\fc 0$ in $E$.
\end{proof}

Recall that a subset $A$ in a normed lattice $(E,\lvert \cdot\rVert)$ is said to be almost order bounded if, for any $\epsilon>0$, there is $u_\epsilon\in E_+$ such that $\lvert (|x|-u_\epsilon)^+\rVert=\lvert |x|-u_\epsilon\wedge|x|\rVert\leq\epsilon$ for any $x\in A$. For a given Banach lattice $f$-algebra $E$, one can give the following definition: a subset $A$ of $E$ is called a {\em $f$-almost order bounded} if, for any $\epsilon>0$, there is $u_\varepsilon\in E_+$ such that $\lVert|x|-u_\epsilon|x|\rVert\leq\epsilon$ for any $x\in A$. Similar to Proposition 3.7 \cite{GX}, we give the following work.
\begin{proposition}\label{$f$-almost}
	Let $E$ be a Banach lattice $f$-algebra. If $(x_\alpha)$ is $f$-almost order bounded and $mn$-converges to $x$, then $(x_\alpha)$ converges to $x$ in norm.
\end{proposition}

\begin{proof}
	Assume $(x_\alpha)$ is $f$-almost order bounded net. Then the net $(|x_\alpha-x|)$ is also $f$-almost order bounded. For any fixed $\varepsilon>0$, there exists $u_\varepsilon>0$ such that
	$$
	\lvert \lvert x_\alpha-x\rvert-u_\epsilon\lvert x_\alpha-x\rvert\rVert\leq\epsilon.
	$$
	Since $x_\alpha\mc x$, we have $\lvert \lvert x_\alpha-x\rvert u_\varepsilon\rVert\to 0$. Therefore, we get $\lvert x_\alpha-x\rVert\leq\varepsilon$, i.e., $x_\alpha\to x$ in norm.
\end{proof}

\begin{proposition}
	In an order continuous Banach lattice $f$-algebra, every $f$-almost order bounded $mo$-Cauchy net converges $mn$- and in norm to the same limit.
\end{proposition}

\begin{proof}
	Assume a net $(x_\alpha)$ is $f$-almost order bounded and $mo$-Cauchy. Then the net $(x_\alpha-x_{\alpha'})$ is $f$-almost order bounded and is $mo$-convergent to zero. Thus, it $mn$-converges to zero by order continuity. Hence, by applying Proposition \ref{$f$-almost}, we get the net $(x_\alpha-x_{\alpha'})$ converges to zero in the norm. It follows that the net $(x_\alpha)$ is norm Cauchy, and so it is norm convergent. As a result, we have that $(x_\alpha)$ $mn$-converges to its norm limit by Remark \ref{rem 1 basic}$(ii)$.
\end{proof}

The multiplication in $f$-algebra is $mo$-continuous in the following sense.
\begin{theorem}
	Let $E$ be a Banach lattice $f$-algebra, and $(x_\alpha)_{\alpha \in A}$ and $(y_\beta)_{\beta \in B}$ be two nets in $E$. If $x_\alpha\mc x$ and $y_\beta\mc y$ for some $x,y\in E$ and each positive element of $E$ can be written as a multiplication of two positive elements then we have $x_\alpha y_\beta\mc xy$.
\end{theorem}

\begin{proof}
	Assume $x_\alpha\mc x$ and $y_\beta\mc y$. Then $\lvert x_\alpha-x\rvert u\nc 0$ and $\lvert y_\beta-y\rvert u\nc 0$ for every $u\in E_+$. Let's fix $u\in E_+$ and $\varepsilon>0$. So, there exist indexes $\alpha_0$ and $\beta_0$ such that $\lVert\lvert x_\alpha-x\rvert u\rVert\leq \varepsilon$ and $\lVert\lvert y_\beta-y\rvert u\rVert\leq\varepsilon$ for all $\alpha\geq\alpha_0$ and $\beta\geq\beta_0$. Next, we show the $mn$-convergence of $(x_\alpha y_\beta)$ to $xy$. By considering the equality $\lvert xy\rvert =\lvert x\rvert \lvert y\rvert $, we have
	\begin{eqnarray*}
		\lVert\lvert x_\alpha y_\beta-xy\rvert u\rVert&=&\lVert\lvert x_\alpha y_\beta-x_\alpha y+x_\alpha y-xy\rvert u\rVert\\&\leq&  \lVert \lvert x_\alpha-x+x\rvert\lvert y_\beta-y\rvert u\rVert+\lVert\lvert x_\alpha -x\rvert \ (\lvert y\rvert u)\rVert\\&\leq& \lVert\lvert x_\alpha-x\rvert \lvert y_\beta-y\rvert u\rVert+\lVert\lvert y_\beta-y\rvert (\lvert x\rvert u)\rVert+\lVert\lvert x_\alpha -x\rvert (\lvert y\rvert u)\rVert.
	\end{eqnarray*}
	The second and the third terms in the last inequality both order converge to zero as $\beta\to\infty$ and $\alpha\to \infty$ respectively because of $\lvert x\rvert u,\lvert y\rvert u\in E_+$, $x_\alpha\mc x$ and $y_\beta\mc y$. Now, let's show the $mn$-convergence of the first term of last inequality. There are two positive elements $u_1,u_2\in E_+$ such that $u=u_1u_2$ because the positive element of $E$ can be written as a multiplication of two positive elements. So, we get $\lVert\lvert x_\alpha-x\rvert \lvert y_\beta-y\rvert u\rVert=\lVert(\lvert x_\alpha-x\rvert u_1)(\lvert y_\beta-y\rvert u_2)\rVert\leq(\lVert\lvert x_\alpha-x\rvert u_1\rVert)(\lVert\lvert y_\beta-y\rvert u_2\rVert)$. Therefore, we see $\lvert x_\alpha-x\rvert \lvert y_\beta-y\rvert u\nc 0$. Hence, we get $x_\alpha y_\beta\mc xy$.
\end{proof}

We give some basic notions motivated by their analogies from vector lattice theory.
\begin{definition}\label{$mn$-notions}
	Let $(x_\alpha)_{\alpha \in A}$ be a net in a Banach lattice $f$-algebra $E$. Then 
	\begin{enumerate}
		\item[(1)] $(x_\alpha)$ is said to be {\em $mn$-Cauchy} if the net $(x_\alpha-x_{\alpha'})_{(\alpha,\alpha') \in A\times A}$ $mn$-converges to $0$,
		\item[(2)] $E$ is called {\em $mn$-complete} if every $mn$-Cauchy net in $E$ is $mn$-convergent,
		\item[(3)] $E$ is called {\em $mn$-continuous} if $x_\alpha\oc0$ implies that $x_\alpha\mc 0$,
	\end{enumerate}	
\end{definition}

\begin{lemma}\label{order and mn-convergence}
	A Banach lattice $f$-algebra is $mn$-continuous iff $x_\alpha\downarrow 0$ implies $x_\alpha\mc 0$.
\end{lemma} 

\begin{proof}
	We show $mn$-continuity. Let $(x_\alpha)$ be an order convergent to zero net in a Banach lattice $f$-algebra $E$. Then there exists a net $z_\beta\downarrow 0$ in $E$ such that, for any $\beta$ there exists $\alpha_\beta$ so that $\lvert x_\alpha\rvert\leq z_\beta$, and so $\lVert x_\alpha\rVert\leq \lVert z_\beta\rVert$ for all $\alpha\geq\alpha_\beta$. Since $z_\beta\downarrow 0$, we have $z_\beta\mc 0$, i.e., for fixed $\varepsilon>0$, there is $\beta_0$ such that $\lVert z_\beta\rVert<\varepsilon$ for all $\beta\geq\beta_0$. Thus, there exists an index $\alpha_{\beta_0}$ so that $\lVert x_\alpha\rVert\leq \varepsilon$ for all $\alpha\geq\alpha_{\beta_0}$. Hence, $x_\alpha\mc 0$.
\end{proof}

\begin{theorem}\label{of-contchar}
	Let $E$ be an $mn$-complete Banach lattice $f$-algebra. Then the following statements are equivalent:
	\begin{enumerate}
		\item[(i)] $E$ is $mn$-continuous;
		\item[(ii)] if $0\leq x_\alpha\uparrow\leq x$ holds in $E$ then $(x_\alpha)$ is an $mn$-Cauchy net;
		\item[(iii)] $x_\alpha\downarrow 0$ implies $x_\alpha\mc 0$ in $E$.
	\end{enumerate}	
\end{theorem}

\begin{proof} 
	$(i)\Rightarrow(ii)$ Take the net $0\leq x_\alpha\uparrow\leq x$ in $E$. Then there exists a net $(y_\beta)$ in $E$ such that $(y_\beta-x_\alpha)_{\alpha,\beta}\downarrow 0$; see Lemma 12.8 \cite{ABPO}. Thus, by applying Lemma \ref{order and mn-convergence}, we have $(y_\beta-x_\alpha)_{\alpha,\beta}\mc 0$. Therefore, the net $(x_\alpha)$ is $mn$-Cauchy because of $\lVert x_\alpha-x_{\alpha'} \rVert_{\alpha,\alpha'\in A}\leq\lVert x_\alpha-y_\beta\rVert+\lVert y_\beta-x_{\alpha'}\rVert$.
	
	$(ii)\Rightarrow(iii)$ Put $x_\alpha\downarrow 0$ in $E$ and fix arbitrary $\alpha_0$. Thus, we have $x_\alpha\leq x_{\alpha_0}$ for all $\alpha\geq\alpha_0$, and so we can get $0\leq(x_{\alpha_0}-x_\alpha)_{\alpha\geq\alpha_0}\uparrow\leq x_{\alpha_0}$. Then it follows from $(ii)$ that the net $(x_{\alpha_0}-x_\alpha)_{\alpha\geq\alpha_0}$ is $mn$-Cauchy, i.e., $(x_{\alpha^{'}}-x_\alpha)\mc 0$ as $\alpha_0\le\alpha,\alpha^{'}\to\infty$. Since $E$ is $mn$-complete, there exists $x\in E$ satisfying $x_\alpha\fc x$ as $\alpha_0\le\alpha\to\infty$. It follows from Proposition \ref{mn convergence positive} that $x_\alpha\downarrow0$ because of $x_\alpha\downarrow$ and $x_\alpha\mc 0$, and so we have $x=0$. Therefore, we get $x_\alpha\mc 0$. 
	
	$(iii)\Rightarrow(i)$ It is just the implication of Lemma \ref{order and mn-convergence}.
\end{proof}

\begin{corollary}\label{of + f implies o}
	Every $mn$-continuous and $mn$-complete Banach lattice $f$-algebra is order complete. 
\end{corollary}

\begin{proof}
	Suppose $E$ is $mn$-continuous and $mn$-complete a Banach lattice $f$-algebra. For $y\in E_+$, put a net $0\leq x_\alpha\uparrow\leq y$ in $E$. By applying Theorem \ref{of-contchar} $(ii)$, the net $(x_\alpha)$ is $mn$-Cauchy. Thus, there exists an element $x\in E$ such that $x_\alpha\mc x$ because of $mn$-completeness.  Since $x_\alpha\uparrow$ and $x_\alpha\fc x$, it follows from Lemma \ref{mn convergence positive} that $x_\alpha\uparrow x$. Therefore, $E$ is order complete.
\end{proof}

The following proposition is an $mn$-version of \cite[Prop.4.2]{GX}.
\begin{proposition}
	Let $E$ be an $mn$-continuous and $mn$-complete Banach lattice $f$-algebra. Then every $f$-almost order bounded and order Cauchy net is $mn$-convergent.
\end{proposition}

\begin{proof}
	Let $(x_\alpha)$ be an $f$-almost order bounded $o$-Cauchy net. Then the net $(x_\alpha-x_{\alpha'})$ is $f$-almost order bounded and is order convergent to zero.	Since $E$ is $mn$-continuous, $x_\alpha-x_{\alpha'}\mc 0$. By using Proposition \ref{$f$-almost}, we have  $x_\alpha-x_{\alpha'}\nc 0$. Hence, we get that $(x_\alpha)$ is $mn$-Cauchy, and so it is $mn$-convergent because of $mn$-completeness.
\end{proof}

We now turn our attention to a topology on Banach lattice $f$-algebras. We show that $mn$-convergence in a Banach lattice $f$-algebra is topological. While $mo$- and $uo$-convergence need not be given by a topology, it was observed in \cite{DOT} that $un$-convergence is topological. Motivated from that definition, we give the following. Let $\varepsilon>0$ be given. For a non-zero positive vector $u\in E_+$, we put
$$
V_{u,\varepsilon}=\{x\in E:\lVert\lvert x\rvert u\lVert<\varepsilon\}.
$$
Let $\mathcal{N}$ be the collection of all the sets of this form. We claim that $\mathcal{N}$ is a base of neighborhoods of zero for some Hausdorff linear topology. It is obvious that $x_\alpha\mc0$ iff
every set of $\mathcal{N}$ contains a tail of this net, hence the $mn$-convergence is the convergence induced by the mentioned topology.

We have to show that $\mathcal{N}$ is a base of neighborhoods of zero. To show this we apply Theorem 3.1.10 \cite{R}. First, note that every element in $\mathcal{N}$ contains zero. Now, we show that for every two
elements of $\mathcal{N}$, their intersection is again in $\mathcal{N}$. Take any two set $V_{u_1,\varepsilon_1}$ and $V_{u_2,\varepsilon_2}$ in $\mathcal{N}$. Put $\varepsilon=\varepsilon_1\wedge\varepsilon_2$ and $u=u_1\vee u_2$. We show that $V_{u,\varepsilon}\subseteq V_{u_1,\varepsilon_1}\cap V_{u_2,\varepsilon_2}$. For any $x\in V_{u,\varepsilon}$, we have $\lVert \lvert x\rvert u\lVert<\varepsilon$. Thus, it follows from $ \lvert x\rvert u_1\leq  \lvert x\rvert u$ that
$$
\lVert\lvert x\rvert u_1\lVert\leq\lVert\lvert x\rvert u\lVert<\varepsilon\leq\varepsilon_1.
$$
Thus, we get $x\in V_{u_1,\varepsilon_1}$. By a similar way, we also have $x\in V_{u_2,\varepsilon_2}$.

It is not a hard job to see that $V_{u,\varepsilon}+V_{u,\varepsilon}\subseteq V_{u,2\varepsilon}$, so that for each $U\in \mathcal{N}$, there is another $V\in \mathcal{N}$ such that $V+V\subseteq U$. In addition, one can easily verify that for every $U\in \mathcal{N}$ and every scalar $\lambda$ with $\lvert \lambda\rvert\leq 1$, we have $λU\subseteq U$.

Now, we show that for each $U\in \mathcal{N}$ and each $y\in U$, there exists  $V\in \mathcal{N}$ with $y+V\subseteq U$. Suppose $y\in V_{u,\varepsilon}$. We should find $\delta>0$ and a non-zero $v\in E_+$ such that $y+V_{v,\delta}\subseteq V_{u,\varepsilon}$. Take $v:=u$. Hence, since $y\in V_{u,\varepsilon}$, we have $\lVert\lvert y\rvert u\rVert<\varepsilon$. Put $\delta:=\varepsilon-\lVert\lvert y\rvert u\rVert$.
We claim that $y+V_{v,\delta}\subseteq V_{u,\varepsilon}$. Let's take $x\in V_{v,\delta}$. We show that $y+x\in V_{u,\varepsilon}$. Consider the inuality $\lvert y+x\rvert u\leq\lvert y\rvert u+\lvert x\rvert u$. Then we have
$$
\lVert\lvert y+x\rvert u\rVert\leq \lVert\lvert y\rvert u\rVert+\lVert\lvert x\rvert u\rVert<\lVert\lvert y\rvert u\rVert+\delta=\varepsilon.
$$

Finally, we show that this topology is Hausdorff. It is enough to show that $\bigcap\mathcal{N}=\{0\}$. Suppose that it is not hold true, i.e., assume that $0\neq x\in  V_{u,\varepsilon}$ for all non-zero $u\in E_+$ and for all $\varepsilon>0$. In particular, take $x\in  V_{\lvert x\rvert,\varepsilon}$. Thus, we have $\lVert \lvert x\rvert^2\rVert<\varepsilon$. Since $\varepsilon$ is arbitrary, we get $\lvert x\rvert^2=0$, i.e., $x=0$ by using Theorem 142.3 \cite{Za}; a contradiction.

Similar to Lemma 2.1 and Lemma 2.2 \cite{KMT}, we have the following two lemmas. The proofs are analogous so that we omit them.
\begin{lemma}\label{mn-top 1}
	$V_{u,\varepsilon}$ is either contained in $[-u,u]$ or contains a non-trivial ideal.
\end{lemma}

\begin{lemma}\label{mn-top 2}
	If $V_{u,\varepsilon}$ is contained in $[-u,u]$, then $u$ is a strong unit.
\end{lemma}

\begin{proposition}
	Let $E$ be a Banach lattice $f$-algebra. Then if a neighborhood of $mn$-topology is norm bounded, then $E$ has a strong unit.
\end{proposition}

\begin{proof}
	Assume $V_{u,\varepsilon}$ is norm bounded for some $u\in E_+$ and $\varepsilon>0$. By applying Lemma \ref{mn-top 1}, we have $V_{u,\varepsilon}$ is contained in $[-u,u]$. So, by using Lemma \ref{mn-top 2}, $u$ is a strong unit.
\end{proof}

Consider Example \ref{int exam}. Then we know that $(I_x,\lVert\cdot\rVert_\infty)$ is a Banach lattice $f$-algebra. $I_x$ equipped with the norm $\lVert\cdot\rVert_\infty$ is lattice isometric to $C(K)$ for some compact Hausdorff space $K$, with $x$ corresponding to the constant one function $\mathsf{1}$; see for example Theorem 3.4 and Theorem 3.6. If $x$ is a strong unit in $E$ then $I_x=E$. It is easy to see that in this case $\lVert\cdot\rVert_\infty$ is equivalent to the original norm and $E$ is lattice and norm isomorphic to $C(K)$. It is easy to see that norm convergence implies $mn$-convergence, and so, in general, norm topology is stronger than $mn$-topology.


\end{document}